\documentclass[12pt,a4paper,psamsfonts]{amsart}

\usepackage{fancyhdr}
\usepackage{appendix}
\usepackage{amssymb,amscd,amsxtra,calc}
\usepackage{mathrsfs}
\usepackage{cmmib57}
\usepackage{multirow}
\usepackage[all]{xy}
\usepackage{longtable}
\usepackage[colorlinks,linkcolor=blue,anchorcolor=blue,citecolor=green]{hyperref}

\usepackage{mathabx,epsfig}

\theoremstyle{plain}
    \newtheorem{thm}{Theorem}[section]

    \newtheorem{corollary}[thm]{Corollary}
    \newtheorem{example}[thm]{Example}
    \newtheorem{lemma}[thm]{Lemma}
    \newtheorem{proposition}[thm]{Proposition}
    \newtheorem{question}[thm]{Question}
    \newtheorem{theorem}[thm]{Theorem}

\theoremstyle{definition}
    \newtheorem{definition}[thm]{Definition}
    \newtheorem{notation}[thm]{Notation}
    \newtheorem*{notation*}{Notation and Terminology}
    \newtheorem{remark}[thm]{Remark}
\theoremstyle{remark}

\newcommand{\PP}{\mathbb{P}}

\newcommand{\Q}{\mathbb{Q}}

\newcommand{\R}{\mathbb{R}}

\newcommand{\aff}{\operatorname{aff}}

\newcommand{\ant}{\operatorname{ant}}
\newcommand{\Aut}{\operatorname{Aut}}
\newcommand{\Bd}{\operatorname{Bd}}
\newcommand{\Bir}{\operatorname{Bir}}

\newcommand{\GL}{\operatorname{GL}}

\newcommand{\NS}{\operatorname{NS}}

\newcommand{\PGL}{\operatorname{PGL}}

\newcommand{\Rk}{\operatorname{Rk_f}}

\newcommand{\Con}{\operatorname{C}}
\newcommand{\J}{\operatorname{J}}
\newcommand{\JN}{\operatorname{J'}}

\newcommand{\Alb}{\operatorname{Alb}}

\begin{document}

\title[Jordan Property]
{Jordan property for automorphism groups of compact varieties}

\author{Yujie Luo, Sheng Meng, De-Qi Zhang}


\address
{
\textsc{Department of Mathematics} \endgraf
\textsc{National University of Singapore,
Singapore 119076, Republic of Singapore
}}
\email{lyj96@nus.edu.sg}

\address{
    \textsc{School of Mathematical Sciences, Ministry of Education Key Laboratory of Mathematics and Engineering Applications \& 
    Shanghai Key Laboratory of PMMP}\endgraf
    \textsc{East China Normal University, Shanghai 200241, China}\endgraf
}
\email{smeng@math.ecnu.edu.cn}

\address
{
\textsc{Department of Mathematics} \endgraf
\textsc{National University of Singapore,
Singapore 119076, Republic of Singapore
}}
\email{matzdq@nus.edu.sg}

\begin{abstract}
In this note, we report some recent progress on the Jordan property for (birational) automorphism groups of projective varieties and compact complex varieties.
\end{abstract}



\maketitle
\tableofcontents

\section{Introduction}

The story starts with a celebrated theorem proven by Camille Jordan \cite[Section 40]{Jor78} in 1878. 
This theorem establishes that for any field $k$ of characteristic zero and any positive integer $n$, the linear automorphism group $\GL_n(k)$ of an $n$-dimensional $k$-vector space possesses the {\it Jordan property}:
there exists a Jordan constant $\J = \J(n)$ such that any finite subgroup $H$ of $\GL_n(k)$ contains an abelian 
subgroup $H_1$ with index $[H : H_1]$ not exceeding $\J(n)$.

Researchers subsequently pondered whether other automorphism groups, such as general (non-linear) automorphism groups or even birational automorphism groups of varieties, exhibit the same Jordan property.

More than a century later, we have gained a substantial understanding of the finite subgroups of the Cremona group of rank 2, denoted as ${\rm Cr}_2(k) = \Bir(\mathbb{P}_k^2)$. Notably, when the characteristic of $k$ is zero, the work of Jean-Pierre Serre \cite[Theorem 5.3]{Ser09} confirms that ${\rm Cr}_2(k)$ possesses the Jordan property with $\J\le 2^{10}\cdot3^4\cdot5^2\cdot7$.
Later, Popov explicitly and systematically stated in \cite{Pop11, Pop14} several problems in the general setting involving algebraic geometry, complex geometry and differential geometry.

In this survey, we will primarily highlight recent advancements, following Serre's work. 
The following is an outline of the structure of this article:

In Section \ref{sec: pre}, we will introduce some terminology and notation, and state some fundamental and crucial lemmas related to the Jordan property.
In Section \ref{sec: bir}, we will focus on birational automorphism groups in characteristic 0, summarizing the significant works of  of Bandman, Birkar, Popov, Prokhorov, Shramov and Zarhin. In particular, we demonstrate that Serre's results have been comprehensively generalized to all dimensions.
In Section \ref{sec: aut}, we will elaborate on the Jordan property for automorphism groups of projective varieties.
In Section \ref{sec: pos}, we will delve into the $p$-Jordan and generalized $p$-Jordan properties in positive characteristics.
In Section \ref{sec: fujiki}, we will present two distinct strategies for addressing Fujiki's class $\mathcal{C}$.
Finally, in Section \ref{sec: manifold}, we will enumerate several partial results pertaining to general compact complex manifolds.

For groups acting smoothly or continuously on a real manifold, its Jordan property has be intensively studied and we refer to Mundet i Riera \cite[\S 1]{MiR19} for the excellent survey of the related results.

We refer to \cite{BZ23} and \cite{BZ24}, for two more excellent surveys on the Jordan property of (bimeromorphic) automorphism groups of holomorphic $\mathbb{P}^1$-bundles, and the interrelations between geometric properties
of complex projective varieties or compact Kähler manifolds and the Jordan
property, or the lack of it, of their (bimeromorphic) automorphism groups.

\par \vskip 1pc \noindent
{\bf Acknowledgement.}

The authors would like to thank Yifei Chen, Popov, Prokhorov, Shramov and the referee for valuable suggestions and comments to improve this paper, and Bandman and Zarhin for bringing the valuable survey papers \cite{BZ23} and \cite{BZ24} to our attention.
The first author is supported by Peng Tsu Ann assistant professorship of NUS.
The second author is supported by Shanghai Pilot Program for Basic Research, 
Science and Technology Commission of Shanghai Municipality (No. 22DZ2229014) and a National Natural Science Fund.
The third author is supported by ARF: A-8002487-00-00, of NUS and thanks the organizers of ``Varieties with Boundaries," June 2024, Wakkanai, and Algebraic Geometry seminar organizers of Hiroshima University and Kyoto University, July 2024, for the opportunities of the talks and support.

\section{Preliminaries}\label{sec: pre}
We recall and adopt some notations and easy lemmas from \cite{MZ18} for the convenience of the readers.
Most of the concepts and results are well-known and appeared in early literature with possibly different formulation.

\begin{definition}[cf. {\cite[Definition~2.2]{MZ18}}]\label{defn: rank and number} Given a group $G$, we introduce the following constants:
$$\begin{array}{lll}
&\Bd(G)=\sup\{|F|:|F|<\infty, F\leq G\},\\
&\Rk(G)=\sup\{\Rk(A):|A|<\infty, A \text{ is abelian}, A\leq G\},\\
\end{array}$$
 where $\Rk(A)$ is the {\it minimal number} of generators of a finite abelian group $A$.
 Similarly, we may define these constants for a family of groups.
\end{definition}

\begin{definition}\label{defn: J and JN}
A group $G$ (resp. a family $\mathcal{G}$ of groups) has the {\it Jordan property} if
there exists a Jordan constant $\J = \J(G)$ (resp. $\J = \J(\mathcal{G})$) such that any finite subgroup $H$ of $G$ (resp. $G\in \mathcal{G}$) contains an abelian 
subgroup $H_1$ with index $[H : H_1]$ not exceeding $\J$.
We use $\JN = \JN(G)$ (resp. $\JN = \JN(\mathcal{G}$)) if we assume $H_1 \lhd H$ in the above definition.
\end{definition}

\begin{remark}
Note that $\J(G)\le \JN(G)\le \J(G)!$.
Indeed, for any finite subgroup $H$ and its abelian subgroup $H_1$ with $[H:H_1]\le \J(G)$.
The action of $H$ on $H/H_1$ yields a homomorphism of $H$ to the symmetric group on $[H:H_1]$ letters.
The kernel of this homomorphism is a normal subgroup $H_2$ of $H$, of index dividing $\J(G)!$, and contained in $H_1$.
\end{remark}

The easy observations below are frequently used.

\begin{lemma}[{\cite[Lemma~2.3]{MZ18}}]\label{lem: pre lemma 1}
Consider the exact sequence of groups
$$1 \to G_1 \to G \to G_2 \to 1.$$

\begin{itemize}
\item[(1)] $\J(G)\leq \Bd(G_2)\cdot \J(G_1)$.
\item[(2)] If $G_1$ is finite, then $\J(G_2)\leq \J(G)$.
\item[(3)] If $\Bd(G_1)=1$, then $\J(G)\le \J(G_2)$.
\item[(4)] If $G\cong G_1\times G_2$, then $\J(G)\leq \J(G_1)\cdot \J(G_2)$.
\item[(5)] $\Rk(G)\leq \Rk(G_1)+ \Rk(G_2)$.
\item[(6)] $\J(G)\leq \J(G_2)\cdot \Bd(G_1)^{\Rk(G_2)\cdot \Bd(G_1)}$.
\end{itemize}
\end{lemma}

One should be careful that the extension of two Jordan groups may no longer be Jordan.
This is also the key obstruction making the Jordan problem challenging. The following geometric example illustrates how the reduction of the Jordan property via geometric canonical fibrations is not automatic.

\begin{example}
    Let $X=\mathbb{P}^1\times E$ where $E$ is an elliptic curve.
    Let $G=\Bir(X)$.
    Then we have an exact sequence
    $$1\to G_1\to G\to G_2\to 1$$
    where the map $G \to G_2:=\Bir(X)|_E$ is obtained from the albanese map $X \to \Alb(X) = E$, 
    $G_2 \le \Aut(E)\cong E\rtimes F$  for some finite group $F$ and the kernel (of $G \to G_2$) $G_1\le \PGL(2, k(E))$.
    Clearly, $G_1$ and $G_2$ are Jordan (cf.~Theorem \ref{thm: automorphism projective}). However, $G$ is not Jordan (cf.~Theorem \ref{thm: bir surface}).
\end{example}

\begin{lemma}[cf. {\cite[Lemma~2.4]{MZ18}}]\label{lem: pre lemma 2} 
Below are some important constants.
\item[(1)] (Jordan) Let $\Con_n := \JN(\GL_n(k))$ (cf.~Definition\ref{defn: J and JN}, \cite[Theorem 36.14, pp. 258-262]{CR62}). Every general linear group $\GL_n(k)$ is a Jordan group, so every linear algebraic group is a Jordan group. We have $\Con_n = (n+1)!$ afforded by $S_{n+1}$ when $n\ge 71$; see \cite{Col07} for the exact value of $\Con_n$.
\item[(2)] (Minkowski) $\Bd(\GL_n(K))<\infty$ for any field $K$ which is finitely generated over $\mathbb{Q}$. The bound depends on $n$ and $[K:\mathbb{Q}]$ (cf. \cite[Theorem 5, and \S4.3]{Ser07}).
\item[(3)] $\Rk(T) = \dim T$, when $T$ is an algebraic torus.
\item[(4)] $\Rk(\GL_n(k))=n$ (see \cite[\S 15.4 Proposition]{Hu75} and use (3)).
\item[(5)] $\Rk(A)=2\dim A$, when $A$ is an abelian variety.
\end{lemma}

In this survey, we will mainly focus on automorphism groups and birational automorphism groups of varieties (mostly projective or compact).

\begin{notation}\label{not: aut}
Let $X$ be an algebraic variety.
Denote by $\Aut(X)$ the group of all automorphisms of $X$ and $\Bir(X)$ the group of all birational self-maps of $X$.
When $X$ is a complex space, $\Aut(X)$ (resp. $\Bir(X)$) refers to the group of all biholomorphic (resp.~ bimeromorphic) self-maps of $X$.

Given an $\R$ -Cartier divisor $L$ in a projective variety $X$, denote by 
$$\Aut_{[L]}(X):=\{g\in \Aut(X)\,|\, g^*L\equiv L\}$$
where $g^*L\equiv L$ means numerical equivalence.
Given a class $[\alpha]\in H^{1,1}(X,\mathbb{R})$ on a compact complex manifold $X$, denote by  
$$\Aut_{[\alpha]}(X):=\{g\in \Aut(X)\,|\, g^*[\alpha]=[\alpha]\}.$$
\end{notation}

By using the Hilbert Scheme when $X$ is a projective variety (cf.~\cite{Kol96}),
we have the following natural inclusion
$$\Aut(X)\subseteq \Bir(X)\subseteq \textup{Hilb}(X\times X)$$
via $g\in \Bir(X)\mapsto \Gamma_g\in \textup{Hilb}(X\times X)$ where $\Gamma_g$ is the graph of $g$.
Note that both $\Aut(X)$ and $\Bir(X)$ are open subschemes of $\textup{Hilb}(X\times X)$.
This way, we denote by $\Aut_0(X)$ the {\it neutral component} of $\Aut(X)$.
A similar statement works by using the Douady space when $X$ is a compact complex variety, see \cite{Fuj78}.

However, $\textup{Hilb}(X\times X)$ and also $\Aut(X)$ may have infinitely many connected components.
So $\Aut(X)$ may not be a variety.
Such an example is easy to construct whenever $X$ admits an automorphism of positive entropy (e.g. when $X$ is a product of elliptic curves).
Nevertheless, if we fix a Hilbert polynomial $P$ of a subvariety, such as the diagonal of $X\times X$ (the graph of the identity map), then $\Aut(X)\cap \textup{Hilb}_P(X\times X)$ will be a group variety with only finitely many connected components. 
Here, $\textup{Hilb}_P(X\times X)\subseteq \textup{Hilb}(X\times X)$ is the subscheme of subvarieties with Hilbert polynomial $P$, and it is a projective variety (cf.~\cite[Theorem 1.4]{Kol96}). 
In this way, we obtain the following significant lemma for the reduction to the connected group.

\begin{lemma}[{\cite[Lemma~2.5]{MZ18}}]\label{lem: pre lemma 3}
Let $X$ be a projective variety.
Then there exists a constant $\ell$ (depending on $X$),
such that for any finite subgroup $G\leq \Aut(X)$, we have $[G:G\cap \Aut_0(X)] \le \ell$.
\end{lemma}

\section{Birational automorphism groups in characteristic 0} \label{sec: bir}
In this section, we always work over fields of characteristic $0$ unless otherwise specified.
It is Popov who asked whether the group $\Aut(X)$ (resp. $\Bir(X)$) of all automorphisms (resp. all birational automorphisms) of an algebraic variety $X$ is Jordan
(cf.~\cite[Question 2.30-2.31]{Pop11}).
Popov \cite{Pop14} predicted that $\Aut(X)$ should be Jordan and till now there is no known counter example.
We shall see in the next section, this has been verified when $X$ is projective.
In this section, we focus on $\Bir(X)$.

Popov himself proved that for a projective surface $X$, the group $\Bir(X)$ is Jordan unless
$X$ is birational to the product $\PP^1 \times E$ with $E$ being an elliptic curve. Later, Zarhin confirmed that
$\Bir(\PP^1 \times E)$ is not Jordan (cf.~\cite[\S 2.2]{Pop11}, \cite[Theorem 5.3]{Ser09}, \cite[Theorem 1.2]{Zar14}, \cite[Theorem 1.3]{Zar15}).
In dimension three, Prokhorov and Shramov \cite[Theorem 1.8]{PS18} gave fully classification of threefolds with non-Jordan birational automorphism groups; see also \cite{PS20} for the case of uniruled K\"ahler threefolds.
We summarize as follows.

\begin{theorem}\label{thm: bir surface}
    Let $X$ be an algebraic surface.
    Then $\Bir(X)$ is Jordan if and only if $X$ is not birational to $\PP^1 \times E$ with $E$ an elliptic curve.
\end{theorem}

\begin{theorem}\label{thm: bir 3fold}
    Let $X$ be an algebraic threefold.
    Then $\Bir(X)$ is Jordan if and only if $X$ is neither birational to $E\times \mathbb{P}^2$ with $E$ an elliptic curve, nor to $S\times \mathbb{P}^1$ with $S$ being one of the following:
    \begin{itemize}
        \item an abelian surface;
        \item a bielliptic surface;
        \item a surface of Kodaira dimension 1 such that the Jacobian fibration of the pluricanonical fibration is Zariski locally trivial.
    \end{itemize}
\end{theorem}

\begin{remark}
    If one considers the solvably (resp. nilpotently) Jordan property which replaces the abelian subgroup by the solvably (resp. nilpotent) subgroup, then for complex algebraic variety $X$, the group $\Bir(X)$ is always solvably (resp. nilpotently) Jordan (cf. \cite{PS14, PS16, Gul20}).
\end{remark}

In higher dimensions, with the help of the minimal model program (cf.~\cite{BCHM10}), 
Prokhorov and Shramov \cite[Theorem 1.8]{PS14} (see also \cite[Theorem 1.8]{PS16}) confirmed the Jordan property of the group $\Bir(X)$ for any algebraic variety $X$, assuming that either $X$ is non-uniruled
or $X$ has vanishing irregularity as well as the (then) outstanding
Borisov-Alexeev-Borisov conjecture about the boundedness of terminal Fano varieties
which has now been affirmatively confirmed by Birkar \cite[Theorem 1.1]{Bir21}.
In particular, the result applies when $X$ is rationally connected (i.e., any two points are connected by a rational curve).

We list these results as follows and sketch their ideas.
Recall that, for an algebraic variety $X$, we define 
$$\widetilde{q}(X)=h^1(X',\mathcal{O}_{X'})$$
where $X'$ is any smooth projective birational model of $X$.

\begin{theorem}[{\cite[Theorem 1.8]{PS14} and \cite[Corollary 1.5]{Bir21}}]\label{thm: bir jordan}
    Let $X$ be a rationally connected algebraic variety. 
    Then $\Bir(X)$ is Jordan with the Jordan constant depending only on $\dim X$. 
\end{theorem}

To prove Theorem \ref{thm: bir jordan}, it suffices to show that for any finite group $G$ (faithfully) acting on a terminal rationally connected variety $X$, there is a subgroup $H\le G$ with index $[G:H]\le C$, where $C$ depends only on $\dim X$, such that $H$ has a fixed point on $X$.
Such idea was first introduced by Popov in \cite{Pop14}.
Further reduction argument allows us to focus only on terminal Fano $X$.
This way, there exists some (minimal) $m>0$ such that the anti-pluri-canonical divisor $-mK_X$ is a very ample Cartier divisor and defines an $G$-equivariant embedding $X\hookrightarrow \mathbb{P}^{N}$ where $N=h^0(X, -mK_X) - 1$.
So $G\le \PGL(N+1)$.
Finally, the Borisov-Alexeev-Borisov conjecture, proved by Birkar, tells that such $N$ is upper bounded by a function in $\dim X$.

\begin{remark}
    In particular, all Cremona groups ${\rm Cr}_n(k)=\Bir(\mathbb{P}^n_k)$  have the Jordan property, confirming a conjecture of Serre.
    Here, we list some known results for the Jordan constants.
    \begin{enumerate}
        \item Yasinsky \cite[Theorems 1.9 and 1.10]{Yas17} showed that the normal Jordan constant $\JN({\rm Cr}_2(k))=7200$ when $k$ is algebraically closed of characteristic $0$ and $\JN({\rm Cr}_2(\mathbb{R}))=\JN({\rm Cr}_2(\mathbb{Q}))=120$.
        \item Zaitsev \cite[Theorem 1.3 and Corollary 1.4]{Zai24} fully classified $\JN({\rm Cr}_2(k))$ for any $k$ of characteristic $0$ and there are exactly three values $7200, 168, 120$.
        \item Prokhorov and Shramov \cite{PS23} showed that $\JN({\rm Cr}_2(\mathbb{F}_q))=|\PGL_3(\mathbb{F}_q)|=q^3\cdot(q^2 - 1)\cdot(q^3 - 1) $ when $q\not\in \{2,4,8\}$. 
        Later, Vikulova \cite{Vik23} showed that this is true when $q\neq 2$ and the only exceptional case is $\JN({\rm Cr}_2(\mathbb{F}_2))=|S_6|=720>|\PGL_3(\mathbb{F}_2)|$.
        \item Prokhorov and Shramov \cite[Theorem 1.2.4]{PS17} gave a uniform upper bound $\JN(\Bir(X))\le 107495424$ for any rationally connected threefold $X$.
    \end{enumerate}
\end{remark}

We reorganize the strategy in \cite{PS14} and provide a slightly streamlined proof for the following theorem. 
\begin{theorem}[{cf.~\cite[Theorem 1.8]{PS14}}]\label{thm: bir nu jordan}
    Let $X$ be a non-uniruled variety. 
    Then $\Bir(X)$ is Jordan.
\end{theorem}

\begin{proof}
    A projective variety $Y$ is {\it quasi-minimal} if $Y$ has terminal singularities and the canonical divisor $K_Y$ is movable (cf.~\cite[Definition 4.2]{PS14}).
    Note that $X$ admits a quasi-minimal model by simply running a minimal model program till no divisorial contraction occurs; see the last part of the proof of \cite[Lemma 4.4]{PS14} and take the group $\Gamma$ there to be the trivial group.
    So far, we do not need to consider any finite group action or run any equivariant minimal model program as in \cite[Lemma 4.4]{PS14} and its proof.
    
    Now we may assume that $X$ is quasi-minimal and $\Q$-factorial.
    Denote by $$\textup{PAut}(X):=\{g\in \Bir(X)\,|\, g \textup{ is isomorphic in codimension one}\}$$
    the group of pseudo-automorphisms.
    It is known that $\Bir(X)=\textup{PAut}(X)$, see \cite[Proposition 4.6]{PS14}.
    Consider the exact sequence
    $$(*) \hskip 1pc
    1\to G\to \textup{PAut}(X)\to \textup{PAut}(X)|_{\NS(X)}\to 1$$
    where the restriction $\textup{PAut}(X)|_{\NS(X)}$ is the image of $\textup{PAut}(X)\to \GL(\NS(X))$ induced by the pullback action, and it makes sense because the small birational maps do not affect divisors.
    Now the key observation is that $G\le \Aut(X)$ by noting that any $g\in G$ satisfies $g^*A\equiv A$ for an ample divisor $A$ and the negativity lemma, see \cite[Proposition 2.1]{JM24a} for a more generalized setting.
    Here, we provide a simple algebraic version in Lemma \ref{lem: paut}.
    
    We shall see that $\Aut(X)$ is Jordan by Theorem \ref{thm: automorphism projective}.
    So the theorem is proved by noting that $\textup{PAut}(X)|_{\NS(X)}\le \GL(\NS(X))$ has bounded finite subgroups.
 \end{proof}

\begin{lemma}\label{lem: paut}
    Let $f\in \textup{PAut}(X)$ with $X$ being a $\mathbb{Q}$-factorial normal projective variety over an algebraically closed field $k$ of arbitrary characteristic.
    Suppose $f^*A\equiv A$ for some ample divisor $A$.
    Then $f\in \Aut(X)$.
\end{lemma}

\begin{proof}
    Let $Y$ be the normalization of the graph of $f$.
    Denote by $p_i:Y\to X$ the two induced projections for $i=1,2$.
    Since $f$ is isomorphic in codimension $1$, $p_1$ and $p_2$ share the same exceptional divisors.
    Note that $p_i$ are birational and $f^*A={p_1}_*p_2^*A\equiv A$.
    
    We claim that ${p_1}^*f^*A=p_2^*A$.
    Note that $F:=p_1^*{p_1}_*p_2^*A-p_2^*A$ is an effective $p_1$-exceptional (and hence $p_2$-exceptional) divisor.
    By the assumption, $F\equiv p_1^*A-p_2^*A$ is $p_2$-nef.
    By the negativity lemma (cf.~\cite[Lemma 3.39]{KM98}), $-F$ is effective.
    So $F=0$ and the claim is proved.

    Let $C\subseteq Y$ be any curve such that $p_1(C)$ is a point.
    Then $$A\cdot {p_2}_*C=p_2^*A\cdot C=p_1^*f^*A\cdot C=0$$ by the projection formula.
    Therefore, $p_2(C)$ is a point.
    By the rigidity lemma (cf.~\cite[Lemma 1.15]{Deb01}), $f$ is well-defined.
    Similarly, $f^{-1}$ is well-defined.
\end{proof}

\begin{remark}\label{rmk: nu q bounded}
    In the proof of Theorem \ref{thm: bir nu jordan}, we actually have $G\le \Aut_{[L]}(X)$ (a finite extension of $\Aut_0(X)$ by a classical result due to Fujiki~\cite{Fuj78} and Lieberman~\cite{Lie78} independently) for any ample divisor $L$. If the quasi-minimal $X$ further has $\widetilde{q}(X)=0$, then $\Aut_0(X)$ is trivial by \cite{Fuj78} and \cite{Lie78} and hence $G$ is finite.
    In particular, $\Bir(X)$ has bounded finite subgroups (i.e., $\Bd(\Bir(X))<\infty$) when $X$ is non-uniruled and $\widetilde{q}(X)=0$.
 \end{remark}

In general, we make use of the maximal rationally connected fibration $\varphi:X\dashrightarrow Y$ where $Y$ is non-uniruled and the geometric generic fibre $X_{\overline{k(Y)}}$ of $\varphi$ is rationally connected.
Note that $\Bir(X)$ descends to $Y$ and we have the following exact sequence
     $$1\to G\to \Bir(X)\to \Bir(X)|_Y\to 1$$
where $\Bir(X)|_Y$ is the image of $\Bir(X)\to \Bir(Y)$ and $G\le \Bir(X_{\overline{k(Y)}})$ is Jordan by Theorem \ref{thm: bir jordan}.
Suppose further $\widetilde{q}(X)=0$.
Then $\widetilde{q}(Y)=0$ and hence $\Bir(X)|_Y\le \Bir(Y)$ has bounded finite subgroups by Remark \ref{rmk: nu q bounded}.
Now Lemma \ref{lem: pre lemma 1} implies the following:

\begin{theorem}[{cf.~\cite[Theorem 1.8]{PS14}}]\label{thm: bir q jordan}
    Let $X$ be an algebraic variety with $\widetilde{q}(X)=0$. 
    Then $\Bir(X)$ is Jordan. 
\end{theorem}

\begin{remark}
By Lemma \ref{lem: pre lemma 1}, $\Rk$ is subadditive under group extension. 

Thus one can show that $\Rk(\Bir(X))<\infty$ for any algebraic variety by the previous discussion.

Indeed, via the exact sequence induced by the maximal rationally connected fibration, it suffices to show $\Rk(\Bir(X))<\infty$ when $X$ is quasi-minimal and when $X$ is rationally connected.

When $X$ is quasi-minimal, $\Rk(\Aut(X))<\infty$ and $\Rk(\textup{PAut}(X)|_{\NS(X)})<\infty$ by Theorem \ref{thm: automorphism projective} and Lemma \ref{lem: pre lemma 2}, so $\textup{PAut}(X)$ (and
$\Aut(X)$) have $\Rk$ being finite by the exact sequence $(*)$ in Theorem \ref{thm: bir nu jordan} and the subadditivity of $\Rk$.

When $X$ is rationally connected, by the same $G$-MMP strategy as in the proof of \cite[Lemma 7.6]{PS14}, it suffices to show $\Rk(\Bir(X))<\infty$ when $X$ is terminal Fano. As we have discussed, there exists some $N(n)>0$ such that for any finite group $G\le \Aut(X)$ with $X$ being $n$-dimensional terminal Fano, $G\le \PGL_{N(n)}(k)$.
So $\Rk(\Bir(X))\le \Rk(\PGL_{N(n)}(k))<\infty$ (cf.~Lemma \ref{lem: pre lemma 2}).
Tracing the argument, one sees that $\Rk(\Bir(X))$ is bounded for all $n$-dimensional rationally connected $X$.
\end{remark}

We shall see the progress on positive characteristics in Section \ref{sec: pos}.
\section{Automorphism groups in characteristic 0} \label{sec: aut}
In this section, we always work over fields of characteristic $0$.
The second and third authors proved that $\Aut(X)$ is Jordan when $X$ is a projective variety, which answers the projective case of \cite[\S2, Problem A]{Pop14}.

\begin{theorem}[{\cite[Theorem~1.6]{MZ18}}]\label{thm: automorphism projective}
Let $X$ be a projective variety. Then $\Aut(X)$ is a Jordan group and $\Rk(\Aut(X))<\infty$.
\end{theorem}

The approach in \cite{MZ18} towards Jordan property for the full automorphism group $\Aut(X)$ of a
projective variety $X$ in arbitrary dimension is more algebraic-group theoretical. It does
not use the classification of projective varieties. To be specific, the authors started by proving that every (not necessarily linear) algebraic group has the Jordan property, which is a direct corollary of either one of the following two stronger results.

Let $\mathcal{G}$ be a family of groups. We say $\mathcal{G}$ is {\it uniformly Jordan} if $\{J(G)\}_{G\in \mathcal{G}}$ is upper bounded. 

\begin{theorem}[{\cite[Theorem~1.3]{MZ18}}]\label{thm: automorphism algebraic group 1}
Fix an integer $n \ge 0$.
Let $\mathcal{G}$ be the family of all connected algebraic groups of dimension $n$. Then $\mathcal{G}$ is uniformly Jordan and $\Rk(\mathcal{G})<\infty$.
\end{theorem}

\begin{theorem}[{\cite[Theorem~1.4]{MZ18}}]\label{thm: automorphism algebraic group 2}
Fix an integer $n \ge 0$.
Let $\mathcal{G}=\{\Aut_0(X) \, | \, X$ is a projective variety of dimension $n\}$.
Then $\mathcal{G}$ is uniformly Jordan and $\Rk(\mathcal{G})<\infty$.
\end{theorem}

\begin{remark}\label{nodirect}
Theorem \ref{thm: automorphism algebraic group 2} cannot be deduced directly from Theorem \ref{thm: automorphism algebraic group 1} because the dimension of $\Aut_0(X)$ cannot be bounded solely in terms of $\dim X$. 
For instance, according to \cite[Theorem 3]{Mar71}, 
if $\mathbb{F}_d$ denotes the Hirzebruch surface of degree $d \geq 1$, then the unipotent radical of $\Aut_0(\mathbb{F}_d)$ is $\mathbb{G}_a^{d+1}$, where $\mathbb{G}_a = (k, +)$ represents the additive group. 
Notably, $\dim \Aut_0(\mathbb{F}_d) = d + 5$, which is not bounded (in terms of $\dim \mathbb{F}_d =2$). 
To establish Theorem \ref{thm: automorphism algebraic group 2}, we mitigate the impact of such unipotent radicals through a crucial lemma: \cite[Lemma 3.7]{MZ18}.
\end{remark}

\begin{remark}
    We briefly outline the strategy of the proofs for the aforementioned two theorems. Let $G$ be a connected algebraic group, and let $G_{\aff}$ represent the largest connected affine normal subgroup. 
It is important to note that any anti-affine subgroup of $G$ is connected and lies within the centre of $G$, thereby being normal in $G$. 
    We denote the largest anti-affine subgroup by $G_{\ant}$.
The foundation of our proof lies in the classic decomposition theorem $G = G_{\aff} \cdot G_{\ant}$ for a connected algebraic group $G$. 
This theorem is attributed to Rosenlicht \cite[Corollary 5, p. 440]{Ros56}, and further modern elaborations can be found in Brion \cite{Bri09}.
We also use the effective (and optimal) upper bound in Brion \cite[Proposition 3.2]{Bri13}
for the dimension of the anti-affine part of $\Aut_0(X)$.
For any projective variety $X$ and its normalization $X'$, $\Aut(X)$ lifts to $X'$. So it suffices to prove Theorem~\ref{thm: automorphism projective} when $X$ is normal. Then, Lemma~\ref{lem: pre lemma 3}, Theorem \ref{thm: automorphism algebraic group 1} and an application of \cite[Lemma 3.7]{MZ18} imply
Theorem \ref{thm: automorphism projective}.
\end{remark}

At the end of this section, we list two well known related open questions.
We are only interested in the situation when $\Bir(X)$ is not Jordan, see Theorem \ref{thm: bir q jordan}.
The following question was first formulated by Popov in \cite[Corollary 2.30]{Pop11}.

\begin{question}\label{que: algebraic variety jordan}
    Let $X$ be an algebraic variety.
    Is $\Aut(X)$ Jordan?
\end{question}

Bandman and Zarhin proved that $\Aut(X)$ is Jordan when $\dim X =2$ or $X$ is birational to the product $\PP^1 \times A$ with $A$ a smooth projective variety having no rational curve; see~\cite[Theorem 1.7]{BZ15} and \cite[Theorem 4]{BZ19}.

\begin{question}\label{que: uniform Jordan}
    Let $X$ be a projective variety and $r > 0$.
    Let $\mathcal{G}$ be the family of $\Aut(X')$ with $X'$ birational to $X$ and the Picard number $\rho(X')\le r$.
    Is $\mathcal{G}$ uniformly Jordan?
\end{question}

Bounding the Picard number is necessary by noting that $\Bir(\mathbb{P}^1\times E)$ is not Jordan as in Theorem \ref{thm: bir surface}, and its finite subgroup (action) can always be regularized on some high birational model of $\mathbb{P}^1\times E$.
A natural strategy is to apply Lemma \ref{lem: pre lemma 3}.
However, we do not know whether the $\ell$ there depends only on the Picard number.

%
%
%
\section{Positive characteristic} \label{sec: pos}
%
%
%
%


The Jordan property for $\GL_n(k)$ does not hold in general when $\textup{char}\, k = p > 0$ due to the existence of unipotent elements of finite order. 
For instance, the group $\mathrm{GL}_n(\overline{F}_p)$ contains arbitrarily large subgroups of the form $\mathrm{SL}_n(F_{p^r})$ which are simple modulo their centres. 
Nevertheless, for any finite subgroup $\Gamma$ of $\mathrm{GL}_n(k)$ of order not divisible by $p$, 
there still exists a normal abelian subgroup $A$ of $\Gamma$ with $[\Gamma : A] \leq J(n)$ for some constant $J(n)$. Later, Serre \cite[Theorem 5.3]{Ser09} showed that the Cremona group $\mathrm{Cr}_2(k)$ of rank 2 over a field $k$ also has this property. This motivates the following definition.

\begin{definition}[{\cite[Definition~1.2]{Hu20}}]
Let $p$ be a prime number or zero. A group $G$ is called a \emph{generalized $p$-Jordan group} if there exists a constant $J(G)$, depending only on $G$, such that every finite subgroup $\Gamma$ of $G$ whose order is not divisible by $p$ contains a normal abelian subgroup $A$ of index $\leq J(G)$. We also call such a group generalized Jordan if there is no confusion caused (e.g., in practice, $p$ will always denote the characteristic of the ground field $k$).
\end{definition}

Note that when $p = 0$, this notion coincides with Popov's \cite[Definition 2.1]{Pop11}. The above-mentioned results can be reformulated as follows. Both general linear groups $\mathrm{GL}_n(k)$ and the Cremona group $\mathrm{Cr}_2(k)$ of rank 2 are generalized $p$-Jordan, where $p = \mathrm{char}(k)$. It is proved in \cite{Hu20} that, in addition to the above, any algebraic group is generalized Jordan.

\begin{theorem}[{\cite[Theorem~1.3]{Hu20}}]\label{thm: H20 algebraic group p jordan}
Any algebraic group $G$ defined over a field $k$ of characteristic $p \geq 0$ is generalized $p$-Jordan. Namely, there exists a constant $J(G)$, depending only on $G$, such that every finite subgroup $\Gamma$ of $G(k)$ whose order is not divisible by $p$ contains a normal abelian subgroup of index $\leq J(G).$
\end{theorem}

Another generalization of Jordan’s theorem to prime characteristic was due to Brauer and 
Feit \cite{BF66} by allowing an arbitrary finite subgroup $\Gamma$ of $\mathrm{GL}_n(k)$ whose order may be divisible 
by $p > 0$. They showed that $\Gamma$ contains a normal abelian subgroup whose index is bounded by 
a constant depending on $n$ as well as the order of the $p$-Sylow subgroup $\Gamma^{(p)}$ of $\Gamma$. Larsen and 
Pink subsequently extended Brauer–Feit \cite{BF66} as follows.

\begin{theorem}[{\cite[Theorem~0.4]{LP11}}]\label{thm: H20 linear group p jordan}
For any positive integer $n$, there exists a constant $J'(n)$ such that any finite subgroup $\Gamma$ of $\mathrm{GL}_n(k)$ over a field $k$ of characteristic $p > 0$ contains a 
normal abelian $p'$-subgroup $A$ of index $\leq J'(n) \cdot |\Gamma^{(p)}|^3.$
\end{theorem}

Here a finite group is called a $p$-group (resp. $p'$-group) if its order is some power of $p$ (resp. relatively prime to $p$). In an analogous way, Hu \cite{Hu20} introduced the following notion.

\begin{definition}[{\cite[Definition~1.6]{Hu20}}]\label{defn: p Jordan}
Let $p$ be a prime number. A group $G$ is called a \emph{$p$-Jordan group} if there exist 
constants $J'(G)$ and $e(G)$, depending only on $G$, such that every finite subgroup $\Gamma$ of $G$ contains 
a normal abelian $p'$-subgroup $A$ of index $\leq J'(G) \cdot |\Gamma^{(p)}|^{e(G)}.$
\end{definition}

The following result of Hu extends Theorem~\ref{thm: H20 linear group p jordan} to arbitrary algebraic groups.

\begin{theorem}[{\cite[Theorem~1.7]{Hu20}}]\label{thm: H20 algebraic group p Jordan}
Any algebraic group $G$ defined over a field $k$ of characteristic $p > 0$ is $p$-Jordan. 
That is, there are constants $J'(G)$ and $e(G)$, depending only on $G$, such that any finite subgroup 
$\Gamma$ of $G(k)$ contains a normal abelian $p'$-subgroup $A$ of index $\leq J'(G) \cdot |\Gamma^{(p)}|^{e(G)}.$
\end{theorem}

Together with Lemma \ref{lem: pre lemma 3}, Hu has also deduced two Jordan-type properties for automorphism groups of projective varieties of arbitrary characteristic.

\begin{theorem}[{\cite[Theorems~1.9 and 1.10]{Hu20}}]\label{thm: H20 aut p Jordan}
    Let $X$ be a projective variety defined over a field $k$ of characteristic $p\ge 0$.
    Then $\Aut(X)$ is $p$-Jordan and generalized $p$-Jordan.
\end{theorem}


For birational automorphism groups, an analogue of \cite[Theorem 5.3]{Ser09} is proved by Chen and Shramov.

\begin{theorem}[{\cite[Theorem~1.6]{CS21}}]\label{thm: CS theorem 1}
There exists a constant $J$ such that for every prime $p$ and every field $k$ of characteristic $p$, every finite subgroup $G$ of the birational automorphism group $\operatorname{Bir}(\mathbb{P}^2)$ contains a normal abelian subgroup of order coprime to $p$ and index at most $J \cdot |G_p|^3$, where $G_p$ is a $p$-Sylow subgroup of $G$. In particular, for every field $k$ of characteristic $p > 0$, the group $\operatorname{Bir}(\mathbb{P}^2)$ is $p$-Jordan.
\end{theorem}

Furthermore, in \cite{CS21} an analog of \cite[Theorem 2.32]{Pop11} (see also \cite[Theorem 1.7]{PS21}) is proved.

\begin{theorem}[{\cite[Theorem~1.7]{CS21}}]\label{thm: CS theorem 2}
Let $k$ be an algebraically closed field of characteristic $p > 0$, and let $S$ be an irreducible algebraic surface over $k$. Then the following assertions hold.
\begin{itemize}
    \item[(1)] If $S$ is birational to a product $E \times \mathbb{P}^1$ for some elliptic curve $E$, then the group $\operatorname{Bir}(S)$ is not generalized $p$-Jordan.
    \item[(2)] If the Kodaira dimension of $S$ is negative but $S$ is not birational to a product $E \times \mathbb{P}^1$ for any elliptic curve $E$, then the group $\operatorname{Bir}(S)$ is $p$-Jordan but not Jordan.
    \item[(3)] If the Kodaira dimension of $S$ is non-negative, then the group $\operatorname{Bir}(S)$ is Jordan.
\end{itemize}
\end{theorem}

In the proof of Theorem~\ref{thm: CS theorem 2}, the following assertion is also proved, which makes Hu's result \cite{Hu20} more precise in one important particular case.

\begin{proposition}[{\cite[Proposition~1.8]{CS21}}]\label{prop: CS proposition 3}
Let $k$ be an arbitrary field, and let $X$ be a smooth geometrically irreducible projective variety of non-negative Kodaira dimension over $k$. Then the group $\operatorname{Aut}(X)$ is Jordan.
\end{proposition}

\section{Automorphism groups of K\"ahler manifolds and Fujiki's class $\mathcal{C}$} \label{sec: fujiki}

The result and strategy of \cite{MZ18} are extended to the K\"ahler case by Kim \cite{Kim18}, while Popov offered a much simpler proof by reducing the Jordan property to (connected real) Lie groups.
Note that Boothby and Wang \cite{BW65} proved the Jordan property for connected real Lie groups.
Popov \cite[Theorem 2]{Pop18} generalized it to any real Lie group $G$ with $\Bd(G/G_0)<\infty$.

\begin{theorem}\textup{(\cite{BW65, Pop18})}\label{thm: bw}
    A real Lie group $G$ is Jordan when $\Bd(G/G_0)<\infty$ where $G_0$ is the neutral component of $G$.
\end{theorem}

Recall that a group $G$ is called \emph{strongly Jordan} (cf. \cite{MiRT15}) if it is Jordan, and there exists a constant $m$ depending only on $G$, such that any finite abelian subgroup $A$ of $G$ has rank at most $m$. 

\begin{theorem}\textup{(\cite[Theorem 1.1]{Kim18}, \cite[Theorem 2]{Pop18})}\label{thm: kim-popov}
Let $X$ be a compact K\"ahler space.
Then $\Aut(X)$ is strongly Jordan.
\end{theorem}


People then wonder whether the same conclusion holds for manifolds similar to projective or K\"ahler ones.
We recall some definitions.

A compact complex manifold \(X\) is in \emph{Fujiki's class \(\mathcal{C}\)},
if one of the following three equivalent conditions is satisfied:
\begin{enumerate}
	\item \(X\) is the meromorphic image of a compact K\"ahler manifold.
	\item \(X\) is bimeromorphic to a compact K\"ahler manifold.
	\item \(X\) admits a big \((1,1)\)-class \([\alpha]\).
\end{enumerate}

For the equivalence of the definitions and fundamental properties of Fujiki's class \(\mathcal{C}\), we refer to \cite[Definition 1.1 and Lemma 1.1]{Fuj78},
\cite[Chapter IV, Theorem 5]{Var89}
and \cite[Theorem 0.7]{DP04}.

A compact complex manifold \(X\) is \emph{Moishezon} if $X$ is bimeromorphic to a projective manifold. Hence a Moishezon manifold is in Fujiki's class 
\(\mathcal{C}\). A K\"ahler Moishezon manifold is projective.

Prokhorov-Shramov \cite{PS19} generalised Theorem \ref{thm: kim-popov}
to Moishezon threefolds.
Their idea is to reduce the problem to the projective case.
Roughly speaking, they utilise the maximal rational connected fibration $X \dasharrow V$ which still exists for their Moishezon threefolds (and indeed, for all Moishezon manifolds), and the famous non-uniruled-ness of $V$ due to Graber, Harris and Starr, to either show that $X$ is indeed a projective variety and then apply Theorem \ref{thm: automorphism projective} or show that $X$ is a rationally connected variety and then apply \cite{PS14}, or reduce to a very general fibre (a complex surface or curve) and then apply \cite{PS21}.

Perroni and the last two authors fully proved the Fujiki's case. 

\begin{theorem}[{\cite[Theorem~1.1]{MPZ22}}]\label{thm: MPZ main-thm}
Let $X$ be a connected complex manifold and let $Z\subseteq X$ be a non-empty compact complex subspace in Fujiki's class $\mathcal{C}$.
Then the automorphism group $\Aut(X, Z)$ of all biholomorphic automorphisms of $X$ preserving $Z$ is Jordan.
\end{theorem}

The approach to Theorem \ref{thm: MPZ main-thm} is based on a simple idea: make use of the non-K\"ahler locus of a big $(1,1)$ class $[\alpha]$ on an ``invariant'' subspace $Z$ (in Fujiki's class $\mathcal{C}$)
$$E_{nK}(\alpha):=E_{nK}([\alpha]):=\bigcap\limits_{T\in[\alpha]} \textup{Sing}(T)$$  to find some ``invariant'' K\"ahler submanifold $Z_1 \subseteq Z \subseteq X$. 
Here, the intersection ranges over all K\"ahler currents $T =\alpha+i\partial\bar{\partial}\varphi$ in the class $[\alpha]$, and $\textup{Sing}(T)$ is the complement of the set of points $z\in Z$ such that $\varphi$ is smooth near $z$.
Note that $E_{nK}(\alpha)$ is an analytically closed subvariety; see \cite{Tos18} and \cite{Bou04} for details.
To maintain the ``invariant" property, we may occasionally need to slightly reduce the automorphism group, while the Jordan property remains unaffected.
Next, we focus on the (linear) automorphism group of the normal bundle $\mathcal{N}_{Z_1/X}$ as inspired by Mundet i Riera \cite{MiR19}. 
By employing an equivariant compactification of $\mathcal{N}_{Z_1/X}$, we are able to narrow down the inquiry regarding the Jordan property to the scenario of compact K\"ahler manifolds.

The following result is immediately obtained by taking an equivariant resolution to reduce to the smooth case and then applying Theorem \ref{thm: MPZ main-thm} with $Z=X$.
In particular, it answers the question for the Moishezon manifolds by Prokhorov and Shramov (cf.~\cite{PS19}).

\begin{corollary} [{\cite[Corollary~1.2]{MPZ22}}]\label{cor: MPZCor1}
Let $X$ be a reduced compact complex space.
Then $\Aut(X)$ is strongly Jordan in the following cases (where (1) is a special case of (2)):

\begin{enumerate}
\item $X$ is Moishezon, i.e., $X$ is bimeromorphic to a projective variety.
\item $X$ is in Fujiki's class $\mathcal{C}$, i.e., $X$ is the meromorphic image of a compact K\"ahler manifold.
\end{enumerate}
\end{corollary}


Instead of studying the above invariant subgroup, a more natural way is to study the following question:

\begin{question}\label{Q1}
Let $X$ be a compact complex manifold.
If $X$ is Moishezon (or in Fujiki's class $\mathcal{C}$), can one find a bimeromorphic model $\widetilde{X}$ of $X$
such that $\Aut(X)$ lifts to $\widetilde{X}$, and $\widetilde{X}$ is projective (or K\"ahler)?
\end{question}

Indeed, Demailly and Paun showed in the proof of \cite[Theorem 3.4]{DP04} that if a compact complex manifold $X$ admits a big $(1, 1)$-class $[\alpha]$, then there is a bimeromorphic holomorphic map $\sigma: X' \to X$ from a K\"ahler manifold $X'$ obtained by a sequence of blowups along smooth centres determined by the ideal sheaf $J$ corresponding to some K\"ahler current $T\in [\alpha]$ with analytic singularities (cf.~ \cite[Definition 2.2]{DP04}). 

Generally, the group $\operatorname{Aut}(X)$ may fail to lift via $\sigma$ because the class $[\alpha]$ is not necessarily preserved by $\operatorname{Aut}(X)$. Consequently, we focus on the subgroup 
$\operatorname{Aut}_{[\alpha]}(X)$; see Notation \ref{not: aut}.
However, one still cannot expect the lifting of $\operatorname{Aut}_{[\alpha]}(X)$, because the blown-up ideal sheaf $J$ may not be $\operatorname{Aut}_{[\alpha]}(X)$-invariant. Therefore, we need to find another Kähler model in a more natural way. The idea is to consider the ideal sheaf generated by $g^*J$ for all $g \in \operatorname{Aut}_{[\alpha]}(X)$ and then we are left to verify that the new model admits a K\"ahler form.
In this way, Jia and the second author proved the following:

\begin{theorem}[cf. {\cite[Theorem~1.1]{JM24}}] \label{thm: JM lifting theorem}
     Let $X$ be a compact complex manifold in Fujiki's class $\mathcal{C}$. Then, for any big $(1, 1)$-class $[\alpha]$ on $X$, there exists a bimeromorphic holomorphic map $\sigma: X_e \to X$ from a Kähler manifold $X_e$ such that $\operatorname{Aut}_{[\alpha]}(X)$ lifts holomorphically via $\sigma$.
\end{theorem}

When \(X\) is a K\"ahler manifold with a K\"ahler form \(\alpha\),
Fujiki \cite[Theorem 4.8]{Fuj78} and Lieberman \cite[Proposition 2.2]{Lie78}, separately, proved that
\[
	[\Aut_{[\alpha]}(X):\Aut_0(X)]<\infty.
\]
Their proof heavily relies on the K\"ahler form \(\alpha\) (or at least the existence of a K\"ahler form).
Nevertheless, with the help of Theorem \ref{thm: JM lifting theorem}, we can show the following result.
\begin{corollary}\label{cor: components group}
	Let \(X\) be a compact complex manifold in Fujiki's class \(\mathcal{C}\).
	Then
	\[
		[\Aut_{[\alpha]}(X):\Aut_0(X)]<\infty
	\]
	for any big \((1,1)\)-class \([\alpha]\) on \(X\).
\end{corollary}

Note that Corollary \ref{cor: components group} and Theorem \ref{thm: bw} directly imply Corollary \ref{cor: MPZCor1} with different approaches.
\section{General compact complex manifolds and bimeromorphic groups} \label{sec: manifold}
Let $X$ be a (connected) compact complex manifold.
Recall that we have the following exact sequence induced from the pullback  restriction on the cohomology group
$$1\to \Aut_{\tau}(X)\to \Aut(X)\to \Aut(X)|_{H^2(X,\mathbb{Q})}\to 1$$
where $\Aut_{\tau}(X)$ is defined as the kernel of the surjective map $\Aut(X)\to \Aut(X)|_{H^2(X,\mathbb{Q})}$.
Recall that $\Aut(X)|_{H^2(X,\mathbb{Q})}\le \GL(H^2(X,\mathbb{Q}))$ has bounded finite subgroups.
So it suffices for us to focus on whether $\Aut_{\tau}(X)$ satisfies the Jordan property.
Based on Corollary \ref{cor: components group}, we have the following question.
Note that a positive answer towards Question \ref{que: components group} will directly imply the Jordan property for $\Aut(X)$. 

\begin{question} \label{que: components group}
    Let \(X\) be a compact complex manifold.
    Is $[\Aut_{\tau}(X):\Aut_0(X)]<\infty$?
\end{question}

In non-algebraic cases, compact complex surfaces still behave well. Indeed, by the classification of surfaces, Prokhorov and Shramov \cite[Theorems 1.6 and 1.7]{PS21} showed that $\Aut(X)$ (and even the group $\Bir(X)$) are Jordan for any non-projective compact complex surface $X$.
Nevertheless, Question \ref{que: components group} remains unknown even for surfaces.

\begin{theorem} [cf.~{\cite[Theorems 1.6 and 1.7]{PS21}}]\label{thm: PS theorem 1.7}
Let $X$ be a connected compact complex surface. Then $\Bir(X)$ is not Jordan if and only if $X$ is bimeromorphic to $E \times \mathbb{P}^1$, where $E$ is an elliptic curve. Moreover, there always exists a constant $R = R(X)$ such that every finite subgroup of $\Bir(X)$ is generated by at most $R$ elements.
\end{theorem}





In dimension three, Prokhorov and Shramov \cite{PS22} showed the Jordan property for $\Bir(X)$ when $X$ is a non-uniruled
K\"ahler threefold with non-zero irregularity.
In higher dimensions with large Kodaira dimension, Loginov \cite[Theorem~1.2]{Log25} verified the Jordan property for the bimeromorphic automorphism group.
The proof is based on the study of the bounded finite subgroup property for the pluricanonical representation.

\begin{theorem}[cf.~{\cite[Theorem~1.2]{Log25}}]\label{thm: log main}
Let $X$ be an $n$-dimensional compact complex variety with the Kodaira dimension $\kappa(X) \geq n - 2$. Then the group $\Bir(X)$ is Jordan.
\end{theorem}

Jia introduced the notion of T-Jordan property. 

\begin{definition} ({\cite[Definition~1.2]{Jia23}})
A group \( G \) is called \emph{T-Jordan} (alternatively, we say that \( G \) has the \emph{T-Jordan property}) if there is a constant \( J(G) \) such that every torsion subgroup \( H \) of \( G \) has an abelian subgroup \( H_1 \) with the index \( [H : H_1] \leq J(G) \).
\end{definition}

Lee \cite{Lee76} proved the T-Jordan property for connected real Lie groups.
\begin{theorem}\textup{(\cite{Lee76})}\label{thm: lee}
    A connected real Lie group is T-Jordan.
\end{theorem}

Jia asked the following question:

\begin{question}[cf.~{\cite[Conjecture 1.3]{Jia23}}] \label{que:t_jordan}
	Let \(X\) be a compact complex manifold. Does $\Aut(X)$ satisfy the T-Jordan property?
\end{question}

Jia and Meng confirmed the T-Jordan property for Fujiki's class \(\mathcal{C}\) (cf.~\cite[Corollary~1.5]{JM24}). 

We introduce the following notations as in \cite{Jia23}: Let \(\Xi\) be the set of smooth compact complex surfaces \(X\) in class \uppercase\expandafter{\romannumeral7} with the algebraic dimension \(a(X) = 0\) and the second Betti number \(b_2(X) > 0\). Let \(\Xi_0 \subseteq \Xi\) be those surfaces which have no curve.
Note that \(\Xi_0=\emptyset\) if the Global Spherical Shell conjecture is true and it is known when $b_2(X)=1$ by Teleman \cite{Tel05}.

\begin{proposition}[{\cite[Proposition~1.5]{Jia23}}]
	Let \(X\) be a compact complex surface which is not in \(\Xi_0\).
	Then \(\Aut(X)\) is T-Jordan.
\end{proposition}

\begin{theorem}[{\cite[Theorem~1.7]{Jia23}}] \label{thm: jia vir abelian}
	Let \(X\) be a smooth compact complex surface.
	Then any torsion subgroup \(G\leq \Aut(X)\) admits an abelian subgroup $H$ with finite index $I=[G:H]$.
\end{theorem}

\begin{remark}
    When $X\in \Xi_0$, the index $I$ in Theorem \ref{thm: jia vir abelian} 
    depends on the group $G$ (not just on the surface $X$).
\end{remark}






Note that under settings of non-compact (connected) complex manifolds, or diffeomorphism groups of compact Riemannian manifolds, the Jordan property no longer holds in general.
We refer to \cite{CPS14}, \cite{Pop15}, \cite{Pop18} and \cite{Zar19} for the counter examples; see also \cite{BZ17}, \cite{BZ21}, \cite{MiRT15}, \cite{MiR19} and \cite{MiR20} for positive cases. Here, we highlight a general result by Mundet i Riera.

\begin{theorem}[{\cite[Theorem~1.2]{MiR19}}] \label{thm: MiR}
    Let $X$ be a compact real manifold with non-zero Euler characteristic.
    Then the group of ($C^{\infty}$-)diffeomorphisms of $X$ is Jordan.
\end{theorem}

Again, we refer to Mundet i Riera \cite[\S 1]{MiR19} for the excellent survey on groups acting smoothly or continuously on smooth manifolds.


\begin{thebibliography}{BCHM10}




\bibitem[BZ15]{BZ15}
T.~Bandman and Y.~G.~Zarhin,
Jordan groups and algebraic surfaces,
Transform. Groups \textbf{20} (2015), no. 2, 327-334.

\bibitem[BZ17]{BZ17} T.~Bandman and Y.G.~Zarhin, Jordan groups, conic bundles and abelian varieties, Algebr. Geom. \textbf{4} (2017), no. 2, 229-246.

\bibitem[BZ19]{BZ19}
T.~Bandman and Y.~G.~Zarhin,
Jordan properties of automorphism groups of certain open algebraic varieties,
Transform. Groups \textbf{24} (2019), no. 3, 721-739.

\bibitem[BZ21]{BZ21}
T.~Bandman and Y.~G.~Zarhin,
Bimeromorphic automorphisms groups of certain $\mathbb{P}^1$-bundles,
Eur. J. Math. \textbf{7} (2021), no. 2, 641-670.

\bibitem[BZ23]{BZ23}
T.~Bandman and Y.~G.~Zarhin, Automorphism groups of $\mathbb{P}^1$-bundles over a non-uniruled base,
Russian Mathematical Surveys, \textbf{78} (2023), no. 1, 1–64.

\bibitem[BZ24]{BZ24}
T.~Bandman and Y.~G.~Zarhin,
Jordan Groups and Geometric Properties of Manifolds. Arnold Math J. \textbf{10} (2024), 621–635. 

\bibitem[Bir21]{Bir21}
C.~Birkar,
Singularities of linear systems and boundedness of Fano varieties,
Ann. of Math. (2) \textbf{193} (2021), no. 2, 347-405.

\bibitem[BCHM10]{BCHM10}
C. Birkar, P. Cascini, C. D. Hacon and J. McKernan,
Existence of minimal models for varieties of log general type. J. Amer. Math. Soc., \textbf{23}(2):405-468, 2010.

\bibitem[BW65]{BW65}
W. Boothby and H. Wang,
On the finite subgroups of connected Lie groups,
Comment. Math. Helv. \textbf{39} (1965), 281-294.

\bibitem[Bou04]{Bou04}
S. Boucksom, 
Divisorial Zariski decompositions on compact complex manifolds, 
Ann. Sci. \'Ecole Norm. Sup. (4) \textbf{37} (2004), no. 1, 45-76.

\bibitem[BF66]{BF66} 
R.~Brauer and W.~Feit,
An analogue of Jordan’s theorem in characteristic $p$,
Ann. of Math. (2) \textbf{84} (1966), 119-131.


\bibitem[Bri09]{Bri09} M.~Brion,
On the geometry of algebraic groups and homogeneous spaces,
J. Algebra \textbf{329} (2011), 52-71.

\bibitem[Bri13]{Bri13} 
M.~Brion,
On connected automorphism groups of algebraic varieties,
J. Ramanujan Math. Soc. \textbf{28A} (2013), 41-54.


\bibitem[CS21]{CS21}
Y.~Chen and C.~Shramov,
Automorphisms of surfaces over fields of positive characteristic, 
Geom. Topol. \textbf{28} (2024), no. 6, 2747-2791.

\bibitem[Col07]{Col07}
M. Collins, 
On Jordan's theorem for complex linear groups,
J. Group Theory \textbf{10} (2007), no. 4, 411-423.

\bibitem[CPS14]{CPS14}
B.~Csik{\'o}s, L.~Pyber and E.~Szab{\'o},
Diffeomorphism groups of compact 4-manifolds are not always Jordan,
\href{https://arxiv.org/abs/1411.7524}{arXiv:1411.7524}

\bibitem[CR62]{CR62} C.~W.~Curtis and I.~Reiner,
Representation theory of finite groups and associative algebras,
Pure and Applied Mathematics, Vol. \textbf{XI}. Interscience Publishers, a division of John Wiley \& Sons, New
York-London, 1962.

\bibitem[Deb01]{Deb01}
O. Debarre, {\em Higher-dimensional algebraic geometry}, Universitext, Springer-Verlag,
New York, 2001.

\bibitem[DP04]{DP04}
J.~-P.~Demailly and M.~P\u{a}un,
Numerical characterization of the K\"ahler cone of a compact K\"ahler manifold,
Ann. of Math. (2) \textbf{159} (2004), no. 3, 1247-1274.

\bibitem[Fuj78]{Fuj78}
A.~Fujiki,
On automorphism groups of compact K\"ahler manifolds,
Invent. Math. \textbf{44} (1978), no. 3, 225-258.

\bibitem[Gul20]{Gul20}
A. Guld,
Finite subgroups of the birational automorphism group are `almost' nilpotent of class at most two,
\href{https://arxiv.org/abs/2004.11715v1}{arxiv:2004.11715}

\bibitem[Hu75]{Hu75}
J.~ E.~ Humphreys,
Linear Algebraic Groups,
Graduate Texts in Mathematics, Volume \textbf{21},
Springer, 1975.

\bibitem[Hu20]{Hu20}
F.~Hu, Jordan property for algebraic groups and automorphism groups of projective varieties in arbitrary characteristic, 
Indiana Univ. Math. J. \textbf{69} (2020), no. 7, 
2493-2504.

\bibitem[Jia23]{Jia23}
J.~Jia,
Automorphisms groups of compact complex surfaces: T-Jordan property, Tits alternative and solvability, J. Geom. Anal. \textbf{33} (2023), no. 219, 21 pp.

\bibitem[JM24]{JM24}
J.~Jia and S.~Meng, Equivariant K\"ahler model for Fujiki's class, J. Geom. Anal. \textbf{34} (2024), no. 11, Paper No. 349, 13 pp.

\bibitem[JM24a]{JM24a}
J.~Jia and S.~Meng, Moishezon manifolds with no nef and big classes, Proc. Edinburgh Math. Soc. (to appear), \href{https://arxiv.org/abs/2208.12013}{arXiv:2208.12013}

\bibitem[Jor78]{Jor78}
M. C. Jordan,
M\'emoire sur les \'equations diff\'erentielles lin\'eaires \`a{} int\'egrale alg\'ebrique,
J. Reine Angew. Math. \textbf{84} (1878), 89-215.

\bibitem[Kol96]{Kol96}
J.~Koll\'ar,
Rational curves on algebraic varieties,
Ergeb. Math. Grenzgeb. (3), \textbf{32}
Springer-Verlag, Berlin, 1996. viii+320 pp.

\bibitem[KM98]{KM98}
J.~Koll\'ar and S.~Mori,
{\em Birational geometry of algebraic varieties},
Cambridge Univ. Press, 1998.

\bibitem[Kim18]{Kim18}
J.~Kim,
Jordan property and automorphism groups of normal compact K\"ahler varieties,
Commun. Contemp. Math. \textbf{20} (2018), no. 3, 1750024, 9 pp.

\bibitem[Lie78]{Lie78}
D.~I.~Lieberman,
Compactness of the Chow scheme: applications to automorphisms
and deformations of K\"ahler manifolds,
\emph{Fonctions de plusieurs variables complexes, III}
(\emph{S\'em.\ Fran\c{c}ois Norguet, 1975-1977}), pp.~140-186,
Lecture Notes in Math., \textbf{670}, Springer, Berlin, 1978.

\bibitem[Log25]{Log25}
K.~Loginov, Jordan property for groups of bimeromorphic self-maps of complex manifolds with large Kodaira dimension, Math. Z. \textbf{309}, 21 (2025).

\bibitem[LP11]{LP11}
M.~J.~Larsen and R.~Pink, Finite subgroups of algebraic groups, J. Amer. Math. Soc. \textbf{24}
(2011), no. 4, 1105–1158.

\bibitem[Lee76]{Lee76}
D.~Lee, 
On torsion subgroups of Lie groups, 
Proc. Amer. Math. Soc. \textbf{55} (1976), 3.

\bibitem[Mar71]{Mar71}
M.~Maruyama,
On automorphism groups of ruled surfaces,
J. Math. Kyoto Univ.
\textbf{11}, No. 1 (1971), 89-112.

\bibitem[MPZ22]{MPZ22}
S.~Meng, F.~Perroni  and D.~-Q.~Zhang,
Jordan property for automorphism groups of compact spaces in Fujiki's class ${\mathcal C}$,
J. Topol. \textbf{15} (2022), no. 2, 806-814.

\bibitem[MZ18]{MZ18}
S.~Meng and D.~-Q.~Zhang,
Jordan property for non-linear algebraic groups and projective varieties,
Amer. J. Math. Vol \textbf{140}, no. 4, August 2018, 1133-1145.



\bibitem[MiR19]{MiR19}
I.~Mundet i Riera,
Finite group actions on homology spheres and manifolds with nonzero Euler characteristic,
J. Topol. \textbf{12} (2019), no. 3, 744-758.

\bibitem[MiR20]{MiR20}
I.~Mundet i Riera,
Isometry groups of closed Lorentz 4-manifolds are Jordan,
Geom. Dedicata \textbf{207} (2020), 201-207.

\bibitem[MiRT15]{MiRT15}
I. Mundet i Riera and A. Turull,
Boosting an analogue of Jordan’s theorem for finite groups,
Adv. Math. \textbf{272} (2015), 820–836.


\bibitem[Pop11]{Pop11}
V.~L.~Popov,
On the Makar-Limanov, Derksen invariants, and finite automorphism groups of algebraic varieties. Affine algebraic geometry, 289-311, CRM Proc. Lecture Notes, \textbf{54}, Amer. Math. Soc., Providence, RI, 2011.

\bibitem[Pop14]{Pop14} V.~L.~Popov, Jordan groups and automorphism groups of algebraic varieties, in: Automorphisms in birational and affine geometry, 185-213, Springer Proc. Math. Stat., \textbf{79}, Springer, Cham, 2014.

\bibitem[Pop15]{Pop15}
V.~L.~Popov,
Finite subgroups of diffeomorphism groups,
Proc. Steklov Inst. Math. \textbf{289} (2015), no. 1, 221-226.

\bibitem[Pop18]{Pop18}
V.~L.~Popov,
The Jordan property for Lie groups and automorphism groups of complex spaces,
Math. Notes \textbf{103} (2018), no. 5-6, 811-819.

\bibitem[PS14]{PS14}
Y.~Prokhorov and C.~Shramov,
Jordan property for groups of birational selfmaps,
Compos. Math. \textbf{150} (2014), no. 12, 2054-2072.

\bibitem[PS16]{PS16}
Y.~Prokhorov and C.~Shramov,
Jordan property for Cremona groups,
Amer. J. Math. \textbf{138} (2016), no. 2, 403-418.

\bibitem[PS17]{PS17}
Y.~Prokhorov and C.~Shramov,
Jordan constant for Cremona group of rank 3,
Mosc. Math. J. \textbf{17} (2017), no. 3, 457-509.

\bibitem[PS18]{PS18}
Y.~Prokhorov and C.~Shramov,
Finite groups of birational selfmaps of threefolds,
Math. Res. Lett. \textbf{25} (2018), no. 3, 957-972.

\bibitem[PS19]{PS19}
Y.~Prokhorov and C.~Shramov,
Automorphism groups of Moishezon threefolds,
Mat. Zametki \textbf{106} (2019), no. 4, 636-640; translation in
Math. Notes \textbf{106} (2019), no. 3-4, 651-655.

\bibitem[PS20]{PS20}
Y.~Prokhorov and C.~Shramov,
Finite groups of bimeromorphic selfmaps of uniruled K\"ahler threefolds,
Izv. Ross. Akad. Nauk Ser. Mat. \textbf{84} (2020), no. 5, 169-196; translation in
Izv. Math. \textbf{84} (2020), no. 5, 978-1001

\bibitem[PS21]{PS21}
Y.~Prokhorov and C.~Shramov, Automorphism groups of compact complex surfaces, 
Int. Math. Res. Not. IMRN \textbf{2021}, no. 14, 10490-10520.

\bibitem[PS22]{PS22}
Y.~Prokhorov and C.~Shramov,
Finite groups of bimeromorphic self-maps of nonuniruled K\"ahler threefolds,
Mat. Sb. \textbf{213} (2022), no. 12, 86-108; translation in
Sb. Math. \textbf{213} (2022), no. 12, 1695-1714.

\bibitem[PS23]{PS23}
Y.~Prokhorov and C.~Shramov,
Jordan property for the Cremona group over a finite field,
Tr. Mat. Inst. Steklova \textbf{320} (2023), Algebra, Aritmeticheskaya, Angevraicheskaya i Kompleksnaya Geometriya, 298-310; translation in
Proc. Steklov Inst. Math. \textbf{320} (2023), no. 1, 278-289

\bibitem[Ros56]{Ros56}
M.~Rosenlicht, Some basic theorems on algebraic groups, Amer.\ J. \ Math.
\textbf{78} (1956), 401-443.

\bibitem[Ser07]{Ser07}
J.~P.~Serre, Bounds for the orders of the finite subgroups of $G(k)$, Group representation theory, 405-450, EPFL Press, Lausanne, 2007.

\bibitem[Ser09]{Ser09}
J.~P.~Serre, 
A Minkowski-style bound for the orders of the finite subgroups of the Cremona group of rank 2 over an arbitrary field, 
Moscow Math. J. \textbf{9} (2009), no. 1, 183-198.

\bibitem[Tel05]{Tel05}
A.~Teleman, 
Donaldson theory on non-K\"ahlerian surfaces and class VII surfaces with $b_2=1$, 
Invent. Math. \textbf{162}  (2005), no. 3, 493-521.

\bibitem[Tos18]{Tos18}
V.~Tosatti,
Nakamaye's theorem on complex manifolds,
Algebraic geometry: Salt Lake City 2015, 633-655,
Proc. Sympos. Pure Math., \textbf{97}.1, Amer. Math. Soc., Providence, RI, 2018.

\bibitem[Vik23]{Vik23}
A. Vikulova,
Jordan constant for the Cremona group of rank two over a finite field,
Mat. Zametki \textbf{113} (2023), no. 4, 607-612; translation in
Math. Notes \textbf{113} (2023), no. 3-4, 587-592.

\bibitem[Var89]{Var89}
J.~Varouchas,
K\"ahler spaces and proper open morphisms,
Math. Ann. \textbf{283} (1989), no. 1, 13-52.

\bibitem[Yas17]{Yas17}
E.~Yasinsky,
The Jordan constant for Cremona group of rank 2,
Bull. Korean Math. Soc. \textbf{54} (2017), no. 5, 1859-1871.

\bibitem[Zai24]{Zai24}
A. Zaitsev,
Jordan constants of Cremona group of rank 2 over fields of characteristic zero,
\href{https://arxiv.org/abs/2402.02177v1}{arxiv:2402.02177}


\bibitem[Zar14]{Zar14}
Y.~G.~Zarhin,
Theta groups and products of abelian and rational varieties, 
Proc. Edinb. Math. Soc. (2) \textbf{57} (2014), no. 1, 299-304.

\bibitem[Zar15]{Zar15} 
Y.~G.~Zarhin, 
Jordan groups and elliptic ruled surfaces, Transform. Groups, \textbf{20} (2015), no. 2, 557-572.

\bibitem[Zar19]{Zar19}
Y.~G.~Zarhin, 
Complex Tori, Theta Groups and Their Jordan Properties,
Tr. Mat. Inst. Steklova \textbf{307} (2019), Algebra, Teoriya Chisel i Algebraicheskaya Geometriya, 32-62.



\end{thebibliography}
\end{document}